\documentclass[]{interact}

\usepackage{epstopdf}
\usepackage[caption=false]{subfig}

\usepackage[numbers,sort&compress]{natbib}
\bibpunct[, ]{[}{]}{,}{n}{,}{,}
\renewcommand\bibfont{\fontsize{10}{12}\selectfont}
\makeatletter
\def\NAT@def@citea{\def\@citea{\NAT@separator}}
\makeatother

\theoremstyle{plain}
\newtheorem{theorem}{Theorem}[section]
\newtheorem{lemma}[theorem]{Lemma}

\newtheorem{proposition}[theorem]{Proposition}

\theoremstyle{definition}
\newtheorem{definition}[theorem]{Definition}
\newtheorem{example}[theorem]{Example}

\theoremstyle{remark}
\newtheorem{remark}{Remark}

\newcommand{\R}{\mathbb R}
\newcommand{\N}{\mathbb N}

\newcommand{\X}{\mathbb X}
\newcommand{\Y}{\mathbb Y}
\newcommand{\Uball}{{\mathbb B}}
\newcommand{\Usfer}{{\mathbb S}}

\newcommand{\dom}{{\rm dom}\, }
\newcommand{\grph}{{\rm gph}\,}

\newcommand{\nullv}{\mathbf{0}}

\newcommand{\inte}{{\rm int}\, }

\newcommand{\Lin}{\mathcal{L}}

\newcommand{\GE}{{\rm GE}\,}
\newcommand{\VOP}{{\rm VOP}\,}

\newcommand{\ball}[2]{{\rm B}(#1, #2)}

\newcommand{\Solv}[1]{{\mathcal Sol}(#1)}
\newcommand{\lop}{{\rm sur}\,}

\newcommand{\FDer}{{\rm D}}

\newcommand{\loclop}[2]{{\rm sur}(#1;#2)}
\newcommand{\lip}[2]{{\rm lip}(#1;#2)}
\newcommand{\dist}[2]{{\rm dist}\left(#1,#2\right)}
\newcommand{\exc}[2]{{\rm exc}(#1,#2)}
\newcommand{\incr}[2]{{\rm inc}(#1;#2)}

\newcommand{\Haus}[2]{{\rm Haus}(#1,#2)}
\newcommand{\IESol}[2]{{\rm IE}(#1;#2)}

\begin{document}

\articletype{ }

\title{On some generalized equations with metrically $C$-increasing mappings:
solvability and error bounds with applications to optimization}

\author{
\name{A. Uderzo\textsuperscript{a}\thanks{CONTACT A. Uderzo Email: amos.uderzo@unimib.it}}
\affil{\textsuperscript{a} Department of Mathematics and Applications, Universit\`a
di Milano-Bicocca, Milano, Italy}
}

\maketitle

\begin{abstract}
Generalized equations are problems emerging in contexts of modern
variational analysis as an adequate formalism to treat such issues as
constraint systems, optimality and equilibrium conditions, variational
inequalities, differential inclusions.
The present paper contains a study on solvability and error
bounds for generalized equations of the form $F(x)\subseteq C$,
where $F$ is a given set-valued mapping and $C$ is a closed, convex cone.
A property called metric $C$-increase, matching the metric behaviour of $F$ with the partial
order associated with $C$, is singled out, which ensures solution
existence and error bound estimates in terms of problem data.
Applications to the exact penalization of optimization problems
with constraint systems, defined by the above class of generalized
equations, and to the existence of ideal efficient solutions in vector
optimization are proposed.
\end{abstract}


\section{Introduction}

A feature distinguishing modern variational analysis is the
set-oriented approach to traditional subjects of study.
The language in which many fundamental results are formulated as well as
the comprehensive apparatus of notions and constructions
lying at the core of its theory should make this statement
evident.
Reasons for such a feature can be found in developments of the
last half-century in such areas as optimization and optimal control.
There, with the aim of addressing a broad spectrum of extremum
problems arising in applications, the need to treat
situations falling out from the classical analysis emerged
urgently. As a result, traditional categories of the mathematical
thought such as vectors, equations and functions now share
the center of the stage with new entries, like sets, inequalities,
and multifunctions, considered per se worthwhile of a dedicated
calculus and investigations, not less than the former ones.
Generalized equations, whose study was initiated by S.M.
Robinson (see \cite{Robi79,Robi80}), are problems that seem to be paradigmatic
of this trend. Involving multifunctions and sets in their very structure,
they are powerful enough to subsume standard (equality/inequality)
constraint systems as well as other relations and/or geometric
constraints, various optimality and equilibrium conditions
(often expressed in form of inclusion or inequalities),
variational inequalities and complementarity problems.
So, the study of all the above topics received an effective impulse from
advances in the theory of generalized equations.
Among the central questions in this branch of variational analysis
are solvability and solution stability, error bound estimates,
as well as the sensitivity analysis of solution sets, often leading
to implicit multifunction theorems in the case of parameterized
generalized equations (see, among others, \cite{Robi80}, \cite[Ch. 5]{BonSha00},
\cite[Ch. 5]{BorZhu05}, \cite[Ch. 2 and 4]{DonRoc14},
\cite[Ch. 4.4]{Mord06}).

In this paper, the main subjects of study are solvability and error
bounds for the following type of generalized equations
$$
   \hbox{find $x\in X$ such that $F(x)\subseteq C$},
   \leqno (\GE)
$$
where $F:X\rightrightarrows\Y$ is a given set-valued mapping and
$C\subseteq\Y$ is a nonempty closed and convex subset of a normed space.
This means the study of conditions on the problem data $F$ and $C$,
under which the related solution set, i.e.
$$
   \Solv{F,C}=\{x\in\ X:\ F(x)\subseteq C\},
$$
is nonempty and certain inequalities measuring the distance from
$\Solv{F,C}$ hold true. Whereas a well-developed theory exists for
generalized equations of the form
\begin{equation}   \label{in:alterge}
   \nullv\in f(x)+G(x),
\end{equation}
with $\nullv$ being the null element of $\Y$, $f:X\longrightarrow\Y$
and $G:X\rightrightarrows\Y$ being given data, the case of generalized equations
$(\GE)$ attracted so far a minor interest. Yet, a format like $(\GE)$
emerges in various contexts of optimization theory, as illustrated by the
next example. Besides, in the particular case in which $F$ in $(\GE)$
is a single-valued mapping and it is $G\equiv -C$ in $(\ref{in:alterge})$,
then $(\GE)$ and $(\ref{in:alterge})$ collapse to the same problem.

\begin{example}      \label{ex:GEspecexamples}
(i) Let $f:X\longrightarrow\Y$ be a function taking values in a
vector space $\Y$ partially ordered by its (positive) cone
$C$ and let $R\subseteq X$ be a nonempty set.
Recall that $\bar x\in R$ is said to be an ideally $C$-efficient
solution for the related vector optimization problem
$$
  C\hbox{-}\min f(x) \quad\hbox{ subject to }\qquad x\in R,
  \leqno (\VOP)
$$
provided that
$$
   f(R)\subseteq f(\bar x)+C.
$$
Thus, introducing the set-valued mapping $F:X\rightrightarrows\Y$
defined by $F(x)=f(R)-f(x)$, one gets that the set of all ideally
$C$-efficient solutions coincides with $\Solv{F,C}$. It is worth
recalling that any ideal $C$-efficient solution is, in particular,
also $C$-efficient (see \cite{Jahn04}).

(ii) Essentially, semi-infinite programs are optimization problems,
whose variables lie in a finite-dimensional space, that are subject
to infinitely many constrains (see \cite{HetKor93}). Using $T$ as
an (infinite) index set, a semi-infinite programming problem can be
formalized as follows
$$
    \min_{x\in S}\varphi(x) \quad\hbox{ subject to }\quad
    g(t,x)\le 0,\ \forall t\in T,
$$
where $S\in\R^n$ represents a geometric constraint, $\varphi:\R^n
\longrightarrow\R$ and $g:T\times\R^n\longrightarrow\R$ are
given functions. It is readily
seen that, setting $F(x)=\{g(t,x):\ t\in T\}=g(T,x)$ and $C=
(-\infty,0]$, the feasible region of such kind of problems becomes
$\Solv{F,C}$.

(iii) Consider  a standard constrained convex optimization
problem, namely
$$
  \min_{x\in S}\varphi(x) \quad\hbox{ subject to }\quad
  x\in R,
$$
where $\varphi:\X\longrightarrow\R$ is a convex smooth
functional defined on a normed space $\X$ and $R\subset\X$
is a nonempty, closed and convex set. It is well known
that the (global) optimality of an element $\bar x\in R$ is
characterized by the variational inequality
$$
  \langle\nabla\varphi(\bar x),x-\bar x\rangle\ge 0,\quad\forall
  x\in R,
$$
where $\nabla\varphi(\bar x)$ denotes the G\^ateaux derivative
of $\varphi$ at $\bar x$ (see, for instance, \cite[Proposition 5.1.1]{Shir07}).
Thus, if setting $F(x)=\langle\nabla
\varphi(x),R-x\rangle$ and $C=[0,+\infty)$, $\Solv{F,C}$
coincides with the set of all global solutions to the
above optimization problem.

(iv) In mathematical economics, a production process consists
of the transformation of production factors (scarce resources),
or inputs, into products (goods, services), or outputs.
A production technology is a description of the relationships between
inputs and outputs (see, for instance, \cite{Fare88,Rasm13}).
Such a description may be quantitatively formalized
by a set-valued mapping $F:\R^n\rightrightarrows\R^m$ associating
with each output $x\in\R^n$ the set consisting of all inputs
$y\in\R^m$ needed to produce $x$ (it is reasonable to assume that
the same output can be obtained by means of different combinations
of inputs). Given a cone $C$, any condition like $(\GE)$ can be interpreted
as a constraint on the production technology, corresponding to
specific requirements on the input employment.
\end{example}

Clearly, the study of such issues as solvability and error bounds for
$(\GE)$ can be performed through several approaches. The present
paper proposes an approach, which relies on a property for set-valued
mappings here introduced, called metric $C$-increase. This notion
captures a behaviour of set-valued mappings that combines the partial
order induced by the cone $C$ with metric variations of the given mappings.
Roughly speaking, such property works as a generalization to a
set-valued setting of a decrease principle. In synergy with the Ekeland
variational principle, metric $C$-increase turns out to guarantee
solution existence results for $(\GE)$ as well as quantitative
estimates of the distance from the solution set.
It is worth mentioning that a different approach to the study of error
bound for the above $(\GE)$ has been considered in \cite{Cast98}.
According to such an approach, which rests upon techniques of convex analysis,
the set-valued mapping $F$ is required to satisfy generalized
convexity assumptions, the cone $C$ is required to have nonempty
interior and the solution set to be nonempty. None of these
assumptions will be made according the approach presented here.

The arrangement of the contents of the paper in the subsequent sections
is as follows. In Section 2 some technical preliminaries, adequated
to the analysis to be carried out, are recalled. Essentially, they
deal with translations of set enlargements, with `additive-like'
properties of the excess function and with semicontinuity properties
of set-valued mappings.
In Section 3 the notion of metric $C$-increase for set-valued mappings
is introduced, both in its global and local form, along with the
related exact bounds. Some examples are provided and some connections
with well-known properties in variational analysis are explored.
Section 4 is devoted to exposing the main results of the paper,
which refer to solvability and error bounds for generalized equations
of the form $(\GE)$.
In Section 5 some applications to optimization are discussed:
the first one refers to the existence of ideal efficient solutions
to vector optimization problems; a further application concerns
the exact penalization of general
optimization problems, whose constraint systems are defined by
generalized equations of the form $(\GE)$.


\section{Technical preliminaries}

The notation in use throughout the paper is standard. $\N$ and $\R$
denote the natural and the real number set, respectively. $\R^m_+$
denotes the nonnegative orthant in the space $\R^m$.
Whenever $x$ is an element of a metric space $(X,d)$ and
$r$ is a positive real, $\ball{x}{r}=\{z\in X:\ d(z,x)\le r\}$ denotes
the closed ball with center $x$ and radius $r$.
Given a subset $S$ of a metric
space, by $\inte S$ its topological interior is denoted.
By $\dist{x}{S}=\inf_{z\in S}d(z,x)$ the distance of $x$ from a subset $S\subseteq
X$ is denoted, with the convention that $\dist{x}{\varnothing}
=+\infty$. The $r$-enlargement of a set $S\subseteq X$ is indicated by
$\ball{S}{r}=\{x\in X:\ \dist{x}{S}\le r\}$. Given a pair of subsets
$S_1,\, S_2\subseteq X$, the symbol $\exc{S_1}{S_2}=\sup_{s\in S_1}
\dist{s}{S_2}$ denotes the excess of $S_1$ over $S_2$.
The null vector in a normed space is indicated by $\nullv$. Thus, $\Uball
=\ball{\nullv}{1}$ and $\Usfer=\{x\in\Uball:\ \|x\|=1\}$ stand for
the unit ball and the unit sphere in a normed space, respectively.
Whenever $F:X\rightrightarrows\Y$ is a set-valued mapping, $\grph F$
and $\dom F$ denote the graph and the domain of $F$, respectively.
All the set-valued mappings appearing in the paper will be supposed
to take closed values.
Throughout the text, the acronyms l.s.c. and u.s.c. stand for
lower semicontinuous and upper semicontinuous, respectively.
Further more specific notations will be introduced contextually to their use.

Throughout the current section, $(\Y,\|\cdot\|)$ denotes a real normed space.
For the reader's convenience, in the next remark some basic facts concerning
the interaction between enlargements of a set, translations and the excess
function, which will be employed in the sequel, are collected. Their proof
can be derived as a direct consequence of the definition of the involved
objects.

\begin{remark}     \label{rem:enexdist}
Let $S\subset\Y$ be a nonempty subset and let $r>0$.

(i) It holds
$$
  \bigcup_{y\in S}\ball{y}{r}\subseteq\ball{S}{r}
  \subseteq\bigcap_{\epsilon>0}\bigcup_{y\in S}\ball{y}{r+\epsilon}.
$$
If, in particular, $S$ is a closed subset of a finite-dimensional Euclidean
space, then
$$
   \bigcup_{y\in S}\ball{y}{r}=\ball{S}{r}.
$$

(ii) Let $y\in\Y$. Then, it holds
$$
  \ball{y}{r}+C\subseteq\ball{y+C}{r}\subseteq \bigcap_{\epsilon>0}
  \left(\ball{y}{r+\epsilon}+C\right).
$$
If, in particular, $C$ is a closed subset of a finite-dimensional Euclidean
space, then
$$
  \ball{y+C}{r}=\ball{y}{r}+C.
$$

(iii) Let $S$ be a nonempty subset. It holds
$$
  \ball{S}{r}+C\subseteq\ball{S+C}{r}\subseteq\bigcap_{\epsilon>0}
  \ball{S+C}{r+\epsilon}.
$$
If, in particular, $S$ and $C$ are nonempty subsets of a finite-dimensional
Euclidean space such that $S+C$ is closed, then
$$
  \ball{S}{r}+C=\ball{S+C}{r}.
$$

(iv) Let  $C\subset\Y$ be a closed, convex cone. Then, it results in
$$
   \exc{S}{C}=\exc{S+C}{C}.
$$

(v) Let $S_1$ and $S_2$ be nonempty subsets of $\Y$ and let
$r_1,\, r_2>0$. Then,
$$
  \ball{S_1}{r_1}+\ball{S_2}{r_2}\subseteq\bigcap_{\epsilon>0}
  \ball{S_1+S_2}{r_1+r_2+\epsilon}.
$$
If, in particular, $\Y$ is a finite-dimensional Euclidean space, then
$$
  \ball{S_1}{r_1}+\ball{S_2}{r_2}\subseteq\ball{S_1+S_2}{r_1+r_2}.
$$

(vi) If $C\subset\Y$ is a closed, convex cone, with
$\bar c\in C$ and $y\in \Y\backslash C$, and $\alpha>0$, then
\begin{eqnarray}    \label{in:distCalpha}
  \dist{y+\alpha(y-\bar c)}{C}\ge(1+\alpha)\dist{y}{C}.
\end{eqnarray}
Indeed, from
$$
   \dist {y+\alpha(y-\bar c)}{C}=\inf_{c\in C}
   \|(1+\alpha)y-\alpha\bar c-c\|,
$$
as it is $\alpha\bar c+C\subseteq C$, it follows
\begin{eqnarray*}
   \inf_{c\in C}\|(1+\alpha)y-\alpha\bar c-c\| &\ge&
   \inf_{c\in C}\|(1+\alpha)y-c\|    \\
   &=& \inf_{c\in C}\|(1+\alpha)y-c(1+\alpha)\|
      =(1+\alpha)\dist{y}{C}.
\end{eqnarray*}
\end{remark}

On the basis of some of the facts recalled above, the following
technical lemma, expressing a sort of additivity of the
excess with respect to the radius of balls, can be derived.

\begin{lemma}    \label{lem:excballpoint}
Let  $C\subset\Y$ be a closed, convex cone and
let $y\in\Y\backslash C$. Then, for any $r>0$ it holds
$$
   \exc{\ball{y}{r}}{C}=\dist{y}{C}+r.
$$
\end{lemma}

\begin{proof}
Let us show first that $\exc{\ball{y}{r}}{C}\le\dist{y}{C}+r$.
Let $v\in r\Uball$. From
$$
  \inf_{c\in C}\|y+v-c\|\le\inf_{c\in C}\|y-c\|+\|v\|,
$$
one readily obtains
$$
  \exc{\ball{y}{r}}{C}=
  \sup_{v\in r\Uball}\inf_{c\in C}\|y+v-c\|\le
  \sup_{v\in r\Uball}[\inf_{c\in C}\|y-c\|+\|v\|]=
  \dist{y}{C}+r.
$$
To prove the converse inequality, fix an arbitrary
$\epsilon>0$. Correspondingly, there exists $c_\epsilon\in C$
such that $\|y-c_\epsilon\|<(1+\epsilon)\dist{y}{C}$
(remember that $\dist{y}{C}>0$). Thus, since
$y\ne c_\epsilon$, by recalling inequality $(\ref{in:distCalpha})$
with $\alpha=\|y-c_\epsilon\|^{-1}r$, one obtains
\begin{eqnarray*}
    \exc{\ball{y}{r}}{C} &=& \sup_{v\in r\Uball}
    \inf_{c\in C}\|y+v-c\|\ge \inf_{c\in C}
    \left\|  y+{y-c_\epsilon\over\|y-c_\epsilon\|}
    \, r-c\right\|  \\
    &\ge& \left(1+{r\over\|y-c_\epsilon\|}\right)
    \dist{y}{C}\ge\dist{y}{C}+{r\over 1+\epsilon}.
\end{eqnarray*}
By letting $\epsilon\to 0^+$, one achieves the inequality
$$
  \exc{\ball{y}{r}}{C}\ge\dist{y}{C}+r,
$$
thereby completing the proof.
\end{proof}

The next lemma, which will be employed in the proof of the main
result of the paper, generalizes the previous relation to
enlargements of closed sets.

\begin{lemma}    \label{lem:exbalset}
Let  $C\subseteq\Y$ be a closed, convex cone and let $S\subseteq\Y$
be a nonempty subset such that $\exc{S}{C}>0$. For any $r>0$ it holds
$$
  \exc{\ball{S}{r}}{C}=\exc{S}{C}+r.
$$
\end{lemma}

\begin{proof}
Observe first that
$$
  \exc{S}{C} = \sup_{y\in S}\dist{y}{C}=\sup_{y\in S\backslash C}\dist{y}{C}
  =\exc{S\backslash C}{C}.
$$
Thus, according to Lemma \ref{lem:excballpoint}, in particular one has that
for every $y\in S\backslash C$ it is $\dist{y}{C}=\exc{\ball{y}{r}}{C}-r$.
It follows
\begin{eqnarray*}
   \exc{S}{C} &=& \sup_{y\in S\backslash C}\dist{y}{C}=\sup_{y\in S\backslash C}
    \exc{\ball{y}{r}}{C}-r \\
    &=& \sup_{y\in S\backslash C}\sup_{v\in\ball{y}{r}}\dist{v}{C}-r\le
    \sup_{v\in\ball{S}{r}}\dist{v}{C}-r     \\
    &=&\exc{\ball{S}{r}}{C}-r,
\end{eqnarray*}
where the last inequality holds because, as observed in Remark \ref{rem:enexdist}(i),
it holds
$$
  \bigcup_{y\in S\backslash C}\ball{y}{r}\subseteq\bigcup_{y\in S}\ball{y}{r}
  \subseteq\ball{S}{r}.
$$
Now, in order to establish the converse inequality, fix an arbitrary $\epsilon>0$.
Since as noticed in Remark \ref{rem:enexdist}(i) it is $\ball{S}{r}
\subseteq\cup_{y\in S}\ball{y}{r+\epsilon}$, by applying Lemma
\ref{lem:excballpoint}, one finds
\begin{eqnarray*}
   \exc{\ball{S}{r}}{C} &\le & \exc{\cup_{y\in S}\ball{y}{r+\epsilon}}{C}=
   \sup_{y\in S}\exc{\ball{y}{r+\epsilon}}{C} \\
   &=&\sup_{y\in S}\dist{y}{C}+r+\epsilon
   =\exc{S}{C}+r+\epsilon.
\end{eqnarray*}
The arbitrariness of $\epsilon$ completes the proof.
\end{proof}

Let us point out below a consequence of Lemma \ref{lem:exbalset}
that will be useful in the sequel.

\begin{remark}   \label{rem:notincarr}
Given a closed, convex cone $S\subset\Y$, a nonempty set
$S\subseteq\Y$ and constants $r>0$ and $a>1$, it holds
$$
   \ball{S}{ar}\not\subseteq\ball{S+C}{r}.
$$
Indeed, if assuming that $\ball{S}{ar}\subseteq\ball{S+C}{r}$,
then according to Lemma \ref{lem:exbalset} and Remark
\ref{rem:enexdist}(iv), one would obtain
\begin{eqnarray*}
   \exc{S}{C}+ar &=& \exc{\ball{S}{ar}}{C}\le\exc{\ball{S+C}{r}}{C}   \\
    &=& \exc{S+C}{C}+r=\exc{S}{C}+r,
\end{eqnarray*}
which leads to the evident contradiction $ar\le r$.
\end{remark}

In view of the formulation of the next ancillary result, recall that
a set-valued mapping $F:X\rightrightarrows\Y$ defined on a metric space
is said to be l.s.c. at $\bar x$ if for every open set $O\subseteq\Y$ such that
$F(\bar x)\cap O\ne\varnothing$ there exists $\delta_O>0$ such that
$$
  F(x)\cap O\ne\varnothing,\quad\forall x\in\ball{\bar x}{\delta_O}.
$$
Given a closed convex cone $C\subset\Y$, $F$ is said to be Hausdorff
$C$-u.s.c. at $\bar x$ if for every $\epsilon>0$ there exists
$\delta_\epsilon>0$ such that
$$
  F(x)\subseteq\ball{F(\bar x)+C}{\epsilon}, \quad\forall x\in
  \ball{\bar x}{\delta_\epsilon}.
$$
The next lemma links the above semicontinuity properties of a set-valued
mapping $F$ with the semicontinuity properties of the function
$\phi:X\longrightarrow [0,+\infty)$, defined by
\begin{equation}    \label{def:excfunctFC}
  \phi(x)=\exc{F(x)}{C},
\end{equation}
which is possible to associate with $F$ and $C$. Not surprisingly,
such a function will play a crucial role in the main achievements
in this paper.

\begin{lemma}    \label{lem:s.c.Fphi}
Let $F:X\rightrightarrows\Y$ be a set-valued mapping defined on
a metric space $(X,d)$ and taking values on a normed space $(\Y,\|\cdot\|)$,
and let $C\subset\Y$ be a closed, convex cone.

(i) If $F$ is l.s.c. at $\bar x$, then the function $\phi$ is l.s.c.
at $\bar x$.

(ii) If $F$ is Hausdorff $C$-u.s.c. at $\bar x$, then the function $\phi$
is u.s.c. at $\bar x$.
\end{lemma}

\begin{proof}
(i) Fix $\bar x\in X$. If $F(\bar x)=\varnothing$, then $\phi(\bar x)=
\sup_{y\in\varnothing}\dist{y}{C}=-\infty$. Thus, for any
sequence $(x_n)_{n\in\N}$ in $X$, with $x_n\longrightarrow\bar x$
as $n\to\infty$, the inequality
$$
  \liminf_{n\to\infty}\phi(x_n)\ge -\infty=\phi(\bar x)
$$
is trivially fulfilled.

If $F(\bar x)\ne\varnothing$ and $F(\bar x)\subseteq C$, then
$\phi(\bar x)=0$, so one finds
$$
  \liminf_{n\to\infty}\phi(x_n)\ge 0=\phi(\bar x).
$$

Now, assume that $F(\bar x)\not\subseteq C$, so that
$$
   \phi(\bar x)=\exc{F(\bar x)}{C}=\sup_{y\in F(\bar x)}
   \dist{y}{C}>0.
$$
Fix an arbitrary $\epsilon\in (0,\phi(\bar x))$. Then, corresponding
to $\epsilon$, there exists $y_\epsilon\in\ F(\bar x)$ such that
\begin{eqnarray}     \label{in:distyCeps}
    \dist{y_\epsilon}{C}>\phi(\bar x)-{\epsilon\over 2}.
\end{eqnarray}
Since $F$ is l.s.c. at $\bar x$, there exists $r_\epsilon>0$
such that
$$
   F(x)\cap\inte\ball{y_\epsilon}{\epsilon/2}\ne\varnothing,
   \quad\forall x\in \ball{\bar x}{r_\epsilon}.
$$
Thus, if $(x_n)_{n\in\N}$ is any sequence in $X$, with $x_n
\longrightarrow\bar x$, then for some $n_*\in\N$ one gets
the existence of $y_n\in F(x_n)\cap\inte\ball{y_\epsilon}{\epsilon/2}$
for every $n\ge n_*$.
On the account of inequality $(\ref{in:distyCeps})$, one obtains
\begin{eqnarray*}
    \phi(x_n) \ge  \dist{y_n}{C}\ge\dist{y_\epsilon}{C}
    -d(y_n,y_\epsilon)
     > \phi(\bar x)-\epsilon,\quad\forall n\in\N,\ n\ge n_*,
\end{eqnarray*}
wherefrom it follows
$$
  \liminf_{n\to\infty}\phi(x_n)\ge\phi(\bar x)-\epsilon.
$$
By arbitrariness of $\epsilon$, the thesis follows.

(ii) Let $(x_n)_{n\in\N}$ be any sequence in $X$ converging
to $\bar x$ as $n\to\infty$. Now, it is to be proved that
$$
  \limsup_{n\to\infty}\phi(x_n)\le\phi(\bar x).
$$
If $F(\bar x)=\varnothing$, then $\phi(\bar x)=\sup_{y\in\varnothing}
\dist{y}{C}=+\infty$, so the inequality to be proved trivially holds.
Suppose that $F(\bar x)\ne\varnothing$ and fix $\epsilon>0$.
Since $F$ is Hausdorff $C$-u.s.c. at $\bar x$, corresponding to
$\epsilon$ there exists $\delta_\epsilon>0$ such that
$$
  F(x)\subseteq\ball{F(\bar x)+C}{\epsilon}, \quad\forall x\in
  \ball{\bar x}{\delta_\epsilon}.
$$
As $x_n\longrightarrow\bar x$, there exists $n_*\in\N$ such that
$x_n\in\ball{\bar x}{\delta_\epsilon}$, for every $n\in\N$,
with $n\ge n_*$. Consequently, in force of Lemma \ref{lem:exbalset}
and Remark \ref{rem:enexdist}(iv), one obtains
\begin{eqnarray*}
   \phi(x_n) &=&\exc{F(x_n)}{C}\le\exc{\ball{F(\bar x)+C}{\epsilon}}{C}=
   \exc{F(\bar x)+C}{C}+\epsilon \\
   &=& \exc{F(\bar x)}{C}+\epsilon=\phi(\bar x)+\epsilon,
   \quad\forall n\ge n_*.
\end{eqnarray*}
From the last inequality, it follows
$$
  \limsup_{n\to\infty}\phi(x_n)\le\phi(\bar x)+\epsilon,
$$
which, by arbitrariness of $\epsilon$, completes the proof.
\end{proof}


\section{Metrically $C$-increasing mappings}

Throughout the current section, by $F:X\rightrightarrows\Y$
a set-valued mapping is denoted which is defined on
a metric space $(X,d)$ and takes values in a real normed
space $(\Y,\|\cdot\|)$, partially ordered by a closed, convex cone
$C\subset\Y$. This means that, associated in a standard way with
the cone $C$, a partial order relation $\le_C\subset\Y\times\Y$ is defined
as follows
$$
  y_1\le_C y_2 \qquad\hbox{ iff }\qquad y_2-y_1\in C.
$$

The next definition introduces a property for set-valued mappings
which plays a key role in establishing a solvability result for
generalized equation of the form $(\GE)$, along with related
error bound estimates.

\begin{definition}     \label{def:MCincmap}
(i) A mapping $F:X\rightrightarrows\Y$ is said to be {\it metrically $C$-increasing}
on $X$ if there exists a constant $a>1$ such that
\begin{equation}   \label{in:metincdefglo}
   \forall x\in\ X,\ \forall r>0,\
   \exists u\in\ball{x}{r}:\ \ball{F(u)}{ar}
     \subseteq\ball{F(x)+C}{r}.
\end{equation}
The quantity
$$
    \incr{F}{X}=\sup\{a>1:\ \hbox{ inclusion  $(\ref{in:metincdefglo})$
    holds}\}
$$
is called {\it exact bound of metric $C$-increase} of $F$ on $X$.

(ii) A mapping $F:X\rightrightarrows\Y$ is said to be {\it metrically $C$-increasing}
around $\bar x\in\dom F$ if there exist $\delta>0$ and $a>1$ such that
\begin{equation}   \label{in:metincdefloc}
   \forall x\in\ball{\bar x}{\delta},\ \forall r\in (0,\delta),\
   \exists u\in\ball{x}{r}:\ \ball{F(u)}{ar}
     \subseteq\ball{F(x)+C}{r}.
\end{equation}
The quantity
$$
    \incr{F}{\bar x}=\sup\{a>1:\ \exists\delta>0\ \hbox{such that
    inclusion $(\ref{in:metincdefloc})$ holds}\}
$$
is called {\it exact bound of metric $C$-increase} of $F$ around $\bar x$.
\end{definition}

The meaning of the property introduced above should appear from inclusion
$(\ref{in:metincdefglo})$. Loosely speaking, it postulates that the
values taken by $F$ near a given point `move', among other directions,
towards the cone $C$ in a manner which is proportional, with a
certain rate, to the distance from the given point. The exact
bound of metric $C$-increase provides a quantitative estimate
of such a behaviour.
It should be noticed that this notion is not connected with that of
monotonicity, inasmuch as on the space $X$ any partial
order or vector structure is lacking. The term `metric $C$-increase' indeed
refers to the matching of a metric behaviour with the partial order
induced by $C$ on the range space $\Y$.

The following examples provide concrete classes of mappings fulfilling
the property under study and, at the same time, hint connections with
another property widely investigated in variational analysis. In
what follows, $\Lin(\X,\Y)$ denotes the (Banach) space of all linear bounded
operators between the Banach spaces $\X$ and $\Y$, endowed with
the operator norm. In particular,
$\X^*=\Lin(\X,\R)$.

\begin{example}[Regular linear mappings]    \label{ex:oplinmap}
Recall that a linear bounded operator $\Lambda:\X\longrightarrow\Y$
between Banach spaces is said to be open at a linear rate if there
exists a constant $\alpha>0$ such that
\begin{equation}
  \Lambda\Uball\supseteq\alpha\Uball.
\end{equation}
It is well known that the celebrated Banach-Schauder theorem provides
a characterization for such a property: $\Lambda$ is open at a
linear rate iff $\Lambda$ is onto, i.e. $\Lambda\X=\Y$. The quantity
$$
   \lop\Lambda=\sup\{\alpha>0:\ \Lambda\Uball\supseteq\alpha\Uball\}
$$
is called the exact openness bound of $\Lambda$. Roughly speaking, it gives
a quantitative description of the openness at a linear rate. Such a constant
can be estimated in primal and in dual terms as follows: it holds
$$
  \lop\Lambda={1\over \|\Lambda^{-1}\|^-},
$$
where $\Lambda^{-1}:\Y\rightrightarrows\X$ is the inverse (in general) set-valued
mapping to $\Lambda$, which is always positively homogeneous and convex, while
$\|\cdot\|^{-}$ denotes the inner norm of a positively homogeneous mapping,
namely if $H:\Y\rightrightarrows\X$, it is
$$
  \|H\|^-=\sup_{y\in\Uball}\inf_{x\in H(y)}\|x\|;
$$
besides, it holds
$$
  \lop\Lambda=\inf_{\|u^*\|=1}\|\Lambda^\top u^*\|=\dist{\nullv}{\Lambda^\top\Usfer},
$$
where $\Lambda^\top\in\Lin(\Y^*,\X^*)$ denotes the adjoint operator to $\Lambda$
(see, for instance, \cite[Corollary 1.58]{Mord06}).
Now, let $\Y=\R^m$ be equipped with its classical Euclidean space structure, let
$C=\R^m_+$ and $m\ge 2$. In this setting, consider a linear bounded operator
$\Lambda:\X\longrightarrow\R^m$ such that
$$
  \lop\Lambda>m.
$$
This means that
\begin{equation}     \label{in:mlopLambda}
     \Lambda\Uball\supseteq m\Uball.
\end{equation}
Denote by $\mathcal{B}=\{e^1,\dots,e^m\}$ the canonical base of $\R^m$ and
define
$$
   e=\sqrt{m}\sum_{i=1}^{m}e^i.
$$
It is clear that $e\in m\Uball$. Then, because of inclusion $(\ref{in:mlopLambda})$,
there exists $u\in\Uball$ such that $\Lambda u=e$. On the other hand, notice that
$\ball{e}{\sqrt{m}}\subseteq\R^m_+$. Indeed, if $y\in\ball{e}{\sqrt{m}}$, it is
$\|y-e\|\le\sqrt{m}$, and hence
$$
  |y_i-\sqrt{m}|\le\sqrt{m},\quad\forall i=1,\dots,m,
$$
wherefrom one gets $y_i\ge 0$, for every $i=1,\dots,m$. Thus, it is possible
to deduce the existence of $u\in\Uball$ such that
$$
  \ball{\Lambda u}{\sqrt{m}}=\ball{e}{\sqrt{m}}\subseteq\R^m_+
  \subseteq\Uball+\R^m_+=\ball{\Lambda\nullv+\R^m_+}{1},
$$
where the last equality follows on the account of Remark \ref{rem:enexdist}(iii).
By exploiting the linearity of $\Lambda$, from the last inclusion one
readily obtains that for every $r>0$ there exists $u\in r\Uball$ such that
$$
  \ball{\Lambda u}{r\sqrt{m}}\subseteq\ball{\Lambda\nullv+\R^m_+}{r},
$$
and for every $x\in\X$ and for every $r>0$ there exists $u\in\ball{x}{r}$
such that
$$
  \ball{\Lambda u}{r\sqrt{m}}\subseteq\ball{\Lambda x+\R^m_+}{r}.
$$
Since $\sqrt{m}>1$, the above reasoning shows that any linear
mapping $\Lambda\in\Lin(\X,\R^m)$, with the property that $\lop\Lambda>m$,
turns out to be metrically $\R^m_+$-increasing on $\X$ with exact increase
bound $\incr{\Lambda}{\X}\ge\sqrt{m}$.
\end{example}

\begin{example}     \label{ex:LambdaC}
Let $\Lambda\in\Lin(\X,\R^m)$ and let $C\subset\R^m$ be a closed, convex cone.
Assume that $\incr{\Lambda}{\X}>a$. Then, the set-valued mapping
$L:\X\rightrightarrows\R^m$, defined by
$$
  L(x)=\Lambda x+C
$$
is metrically $C$-increasing and $\incr{L}{\X}\ge a$. Indeed, by
taking into account Remark \ref{rem:enexdist}(ii) and (iii), one finds
that for every $x\in\X$ and $r>0$, for any $\epsilon>0$ it holds
\begin{eqnarray*}
   \ball{L(u)}{ar} &=& \ball{\Lambda u+C}{ar}=\ball{\Lambda u}{ar}+C \\
   &\subseteq & \ball{\Lambda x+C}{r}+C =\ball{\Lambda x+C+C}{r}=\ball{L(x)+C}{r}.
\end{eqnarray*}
In the light of Example \ref{ex:oplinmap}, if for instance $\Lambda\in\Lin(\X,\R^m)$
is such that $\lop\Lambda>m$ and $C=\R^m_+$, then the resulting $L$
is metrically $\R^m_+$-increasing on $\X$ and $\incr{L}{\X}\ge\sqrt{m}$.
\end{example}

\begin{example}[Locally regular nonlinear mapping]
Recall that a mapping $f:X\longrightarrow Y$ between metric spaces is said
to be open at $\bar x\in X$ at a linear rate if there exist positive constants
$\delta$, $\alpha$ and $\zeta$ such that
\begin{equation}     \label{in:loclopdef}
   f(\ball{x}{r})\supseteq\ball{f(x)}{\alpha r}\cap\ball{f(\bar x)}{\zeta},
   \quad\forall x\in \ball{\bar x}{\delta},\ \forall r\in (0,\delta).
\end{equation}
Whenever $f$ is continuous at $\bar x$, inclusion $(\ref{in:loclopdef})$
takes the simpler form
$$
  f(\ball{x}{r})\supseteq\ball{f(x)}{\alpha r},
  \quad\forall x\in \ball{\bar x}{\delta},\ \forall r\in (0,\delta).
$$
The quantity
$$
  \loclop{f}{\bar x}=\sup\{\alpha>0:\ \exists\delta>0\ \hbox{such that
  inclusion $(\ref{in:loclopdef})$ holds}\}
$$
is called the exact bound of local openness of $f$ at $\bar x$. Whenever
$X$ and $Y$ are Banach spaces, according to the Lyusternik-Graves theorem,
a mapping $f:\X\longrightarrow\Y$ strictly differentiable at $\bar x\in\X$
turns out to be open at a linear rate at $\bar x$ iff its strict derivative $\FDer f
(\bar x)\in\Lin(\X,\Y)$ is onto. Moreover, the following primal and dual
estimates hold
$$
   \loclop{f}{\bar x}={1\over \|\FDer f(\bar x)^{-1}\|^-}=
   \inf_{\|u^*\|=1}\|\FDer f(\bar x)^\top u^*\|
$$
(see, for instance, \cite[Theorem 1.57]{Mord06}).
Now, letting $\Y=\R^m$, $C=\R^m_+$ and $m\ge 2$, suppose that a mapping
$f:\X\longrightarrow\R^m$ is continuous at $\bar x$ and such that
$$
  \loclop{f}{\bar x}>m.
$$
This means that there exists $\delta_m>0$ such that
\begin{equation}     \label{in:loclopfmr}
  f(\ball{x}{r})\supseteq\ball{f(x)}{mr},
  \quad\forall x\in\ball{\bar x}{\delta_m},\ \forall r\in (0,\delta_m).
\end{equation}
Fix arbitrarily $x\in\ball{\bar x}{\delta_m}$ and $r\in (0,\delta_m)$,
and define
$$
  e_r=r\sqrt{m}\sum_{i=1}^{m}e^i.
$$
As it is $\|e_r\|=rm$, one has $f(x)+e_r\in\ball{f(x)}{mr}$. Thus, from
inclusion $(\ref{in:loclopfmr})$, it follows that there exists $u\in\ball{x}{r}$
such that
$$
  f(u)=f(x)+e_r,
$$
while from the definition of $e_r$ it follows that $\ball{e_r}{r\sqrt{m}}
\subseteq\R^m_+$. Consequently, one obtains
\begin{eqnarray*}
   \ball{f(u)}{r\sqrt{m}} &=& \ball{f(x)+e_r}{r\sqrt{m}}=f(x)+\ball{e_r}{r\sqrt{m}}
   \subseteq f(x)+\R^m_+ \\
    &\subseteq&  \ball{f(x)+\R^m_+}{r}.
\end{eqnarray*}
This shows that, under the above assumptions, $f$ is metrically
$\R^m_+$-increasing around $\bar x$, with $\incr{f}{\bar x}\ge\sqrt{m}$.
By arguing as in Example \ref{ex:LambdaC}, it is not difficult to see that
the same property is shared with set-valued mappings of the form $f+\R^m_+$.
\end{example}

The next example exhibits a situation in which metric $C$-increase
takes place in the absence of openness.

\begin{example}
Let $X=\Y=\R$ be endowed with its usual Euclidean space structure, let
$C=[0,+\infty)$ and let $f:\R\longrightarrow\R$ be given by
$$
  f(x)=|x|.
$$
The mapping $f$ evidently fails to be open at a linear rate at $\bar x=0$.
Nevertheless, as one easily checks, it turns out to be metrically
$\R_+$-increasing on $\R$ with $\incr{f}{\R}=2$ (in particular, it is
$\R_+$-increasing around $\bar x$).
\end{example}

A feature of the property under study to be pointed out is its robustness
under additive Lipschitz continuous perturbations. Such a stability
behaviour can be exploited to build further examples of metrically $C$-increasing
mappings. In order to formulate this feature, let us
recall that a set-valued mapping $G:X\rightrightarrows Y$ between metric
spaces is said to be Lipschitz continuous on $X$ if there exists a constant
$\beta\ge 0$ such that
$$
  \Haus{G(x_1)}{G(x_2)}\le\beta d(x_1,x_2),\quad\forall x_1,\, x_2\in X,
$$
where $\Haus{S_1}{S_2}$ stands for the Hausdorff-Pompeiu distance of the
sets $S_1$ and $S_2$, i.e.
$$
   \Haus{S_1}{S_2}=\max\{\exc{S_1}{S_2},\,\exc{S_2}{S_1}\}
$$
(on the Lipschitz properties of set-valued mappings see \cite[Ch. 3C]{DonRoc14}).
To avoid some major technicalities arising with translations of set enlargements
in abstract spaces, the statement below is given for mappings taking values in
a finite-dimensional Euclidean space.

\begin{proposition}   \label{pro:addpertmetCinc}
Let $F:X\rightrightarrows\R^m$ and $G:X\rightrightarrows\R^m$ be
set-valued mappings and let $C\subset\R^m$ be a closed, convex cone.
Suppose that:

(i) $F$ is metrically $C$-increasing on $X$;

(ii) $F+G:X\rightrightarrows\R^m$ is closed-valued;

(iii) $G$ is Lipschitz continuous on $X$ with a constant $\beta$
satisfying the condition
$$
   \beta<1-{1\over \incr{F}{X}}.
$$

\noindent Then, the mapping $F+G:X\rightrightarrows\R^m$ is metrically
$C$-increasing on $X$, with
$$
   \incr{F+G}{X}\ge (1-\beta)\incr{F}{X}.
$$
\end{proposition}

\begin{proof}
Notice that hypothesis (iii) implies, in particular, the fact that $\beta<1$.
Fix $x\in X$ and $r>0$. By virtue of the condition on $\beta$ in hypothesis
(iii), it is possible to pick $a\in\ (1,\incr{F}{X})$ in such a way that
\begin{equation}   \label{in:condabeta1}
   a(1-\beta)>1.
\end{equation}
Since $F$ is metrically $C$-increasing on $X$ with exact bound $\incr{F}{X}$,
corresponding to $x$ and $(1-\beta)r>0$ there exists $u\in\ball{x}{r}$ such that
\begin{equation}    \label{in:metCincrFux}
   \ball{F(u)}{a(1-\beta)r}\subseteq\ball{F(x)+C}{(1-\beta)r}.
\end{equation}
Since $G$ is Lipschitz continuous on $X$ with a constant $\beta$,
one has in particular
$$
  \exc{G(u))}{G(x)}\le\beta d(u,x)\le\beta r,
$$
which entails
\begin{equation}    \label{in:GHauslipux}
   G(u)\subseteq\ball{G(x)}{\beta r}.
\end{equation}
By taking into account Remark \ref{rem:enexdist}(iii) and (v), from inclusions
$(\ref{in:GHauslipux})$ and $(\ref{in:metCincrFux})$ one obtains
\begin{eqnarray*}
   \ball{F(u)+G(u)}{a(1-\beta)r} &=&\ball{F(u)}{a(1-\beta)r}+G(u) \\
   &\subseteq& \ball{F(x)+C}{(1-\beta)r}+\ball{G(x)}{\beta r} \\
   &\subseteq& \ball{F(x)+G(x)}{r}.
\end{eqnarray*}
This shows that $F+G$ is metrically $C$-increasing on $X$ and that
$\incr{F+G}{X}\ge a(1-\beta)$. Since the latter remains true for every
$a\in\ (1,\incr{F}{X})$ fulfilling the condition in $(\ref{in:condabeta1})$,
the validity of the estimate in the thesis is also proved.
\end{proof}

\begin{remark}
(i) It is worth noting that the hypothesis (iii) in Proposition \ref{pro:addpertmetCinc}
can be weakened by replacing the Lipschitz continuity of $G$ with a more general
property called Lipschitz $C$-continuity, that postulates the existence of $\beta\ge 0$
such that
$$
  \max\{\exc{G(x_1)}{G(x_2)+C},\,\exc{G(x_2)}{G(x_1)+C}\}\le\beta d(x_1,x_2),
  \quad\forall x_1,\, x_2\in X.
$$

(ii) Proposition \ref{pro:addpertmetCinc} can be easily reformulated for mappings
metrically $C$-increasing around a given point. In such an event, the Lipschitz
continuity of $G$ on $X$ can be replaced with its local counterpart around the
reference point.

(iii) It is well known that a sufficient condition for the sum of two closed
sets to be still closed is that one of them is compact. Therefore, hypothesis
(ii) is satisfied provided that, for instance, $G$ is compact-valued.
\end{remark}

To the aim of better understanding the idea behind the notion of
metric $C$-increase, it is useful to consider how it behaves in the
very special case of single-valued scalar functions, with the cone
$C=(-\infty,0]$. In this setting, in its local version, the metric $C$-increase
of $\varphi:X\longrightarrow\R$ around $\bar x$ prescribes the existence
of $a>1$ such that, if in particular $x=\bar x$,
an element $u\in\ball{\bar x}{r}$ can be found such that
$$
  [\varphi(u)-ar,\varphi(u)+ar]\subseteq\ball{\varphi(\bar x)
  +(-\infty,0]}{r}=\ball{(-\infty,\varphi(\bar x)]}{r}=
  (-\infty,\varphi(\bar x)+r].
$$
This implies that $u\in\ball{\bar x}{r}$ must be such that
$$
   \varphi(u)+ar\le\varphi(\bar x)+r
$$
and therefore it must hold
\begin{equation}      \label{in:decprinc}
   \inf_{x\in\ball{\bar x}{r}}\varphi(x)\le\varphi(\bar x)-(a-1)r.
\end{equation}
The condition expressed by inequality $(\ref{in:decprinc})$ appears in
what in variational analysis is referred to as a decrease principle.
Such kind of statements provides conditions, expressed in terms of
generalized derivatives (even in a metric space setting),
upon which inequality $(\ref{in:decprinc})$
holds true. It found various applications in the analysis of solvability
and error bounds for inequalities defined by l.s.c. scalar functions
(see, among others, \cite[Theorem 3.6.2]{BorZhu05} and \cite[Ch. 1.6]{Peno13}).
For instance, any function $\delta_\varphi:X\longrightarrow [0,+\infty]$
such that for every $x\in X$ and $r,\, c>0$
$$
   \inf_{z\in\ball{x}{r}}\delta_\varphi(z)\ge c\qquad\Longrightarrow\qquad
   \inf_{z\in\ball{x}{r}}\varphi(z)\le\varphi(x)-cr
$$
is called decrease index for $\varphi$. The strong slope of a function
and its various subdifferential representations in Banach space settings
are well-known examples of decrease index. Thus, metric $C$-increase can be
viewed as a condition, directly formulated on set-valued mappings,
which leads to generalize a behaviour that can be obtained through
decrease principles.


\section{Solvability and error bounds}

The main question of this paper, which stimulated the introduction
of the notion of metric $C$-increase, is the solvability of generalized
equation of the form $(\GE)$ and related error bounds. As a
prolegomenon to the analysis of such a question, it is worth
pointing out the following basic topological property of the
solution set $\Solv{F,C}$, which can be obtained at once under mild
assumptions on the problem data.

\begin{proposition}     \label{pro:clsolset}
With reference to a given $(\GE)$, suppose that $F$ that is l.s.c.
on $X$. Then, $\Solv{F,C}$ is a (possibly empty) closed set.
If $F$ is l.s.c. in a neighbourhood of $\bar x\in\Solv{F,C}$,
then $\Solv{F,C}$ is locally closed around $\bar x$, i.e. there
exists $\delta>0$ such that $\Solv{F,C}\cap\ball{\bar x}{\delta}$
is closed.
\end{proposition}

\begin{proof}
Since the convex cone $C$ is closed, it suffices to observe that
$\Solv{F,C}$ can be characterized as a sublevel set of the
function $\phi$ defined as in $(\ref{def:excfunctFC})$, namely
$$
  \Solv{F,C}=\{x\in X:\ \phi(x)\le 0\},
$$
and then to recall Lemma \ref{lem:s.c.Fphi}(i).
\end{proof}

One is now in a position to establish the main findings of the
paper.

\begin{theorem}[Solvability and global error bound]    \label{thm:solgloerbo}
Given a set-valued mapping $F:X\rightrightarrows\Y$, suppose
that:

(i) $(X,d)$ is metrically complete;

(ii) $F$ is l.s.c. on $X$;

(iii) $F$ is metrically $C$-increasing on $X$, with exact bound
$\incr{F}{X}$.

\noindent Then, $\Solv{F,C}\ne\varnothing$ and the following
estimate holds
$$
  \dist{x}{\Solv{F,C}}\le{\exc{F(x)}{C}\over\incr{F}{X}-1},\quad
  \forall x\in X.
$$
\end{theorem}

\begin{proof}
Consider the functional $\phi:X\longrightarrow [0,+\infty)$ defined as
in $(\ref{def:excfunctFC})$ and take an arbitrary $a\in (1,\incr{F}{X})$.
Notice that, by virtue of the lower semicontinuity of $F$ on $X$,
in the light of Lemma \ref{lem:s.c.Fphi}(i), the function $\phi$ is l.s.c.
on $X$ and it is bounded from below by definition. Take an
arbitrary $x_0\in X$. If it is $\phi(x_0)=0$, then, as $C$ is
closed, one has $F(x_0)\subseteq C$, so all the assertions in the thesis
trivially happen to hold true. Assume now that $\phi(x_0)>0$.
Since it is
$$
  \phi(x_0)\le\inf_{x\in X}\phi(x)+\phi(x_0)
$$
and the metric space $(X,d)$ has been assumed to be complete,
according to the Ekeland variational principle (see \cite{Ekel74})
for every $\lambda>0$ there exists $x_\lambda\in X$ such that
$$
   \phi(x_\lambda)\le\phi(x_0),
$$
\begin{eqnarray} \label{in:EVP2}
  d(x_\lambda,x_0)\le\lambda,
\end{eqnarray}
and
\begin{eqnarray}   \label{in:EVP3}
   \phi(x_\lambda)<\phi(x)+{\phi(x_0)\over\lambda}
   d(x,x_\lambda),\quad\forall x\in\ X\backslash
   \{x_\lambda\}.
\end{eqnarray}
Let us show that, if choosing
\begin{eqnarray}    \label{eq:EVPlambda}
   \lambda={\phi(x_0)\over a-1},
\end{eqnarray}
the corresponding $x_\lambda\in X$ is a zero of $\phi$.
Ab absurdo, suppose that $\phi(x_\lambda)>0$. By hypothesis
(ii), corresponding to $r=\phi(x_\lambda)$ there exists
$u\in\ball{x_\lambda}{\phi(x_\lambda)}$ with the property
$$
    \ball{F(u)}{a\phi(x_\lambda)}
     \subseteq\ball{F(x_\lambda)+C}{\phi(x_\lambda)}.
$$
Notice that it must be $u\ne x_\lambda$, because for positive
$r=\phi(x_\lambda)$ the last inclusion
fails to be true, as observed in Remark \ref{rem:notincarr}.
By applying Lemma \ref{lem:exbalset} and Remark \ref{rem:enexdist}(iv),
one obtains
\begin{eqnarray*}
  \phi(u) &=&\exc{F(u)}{C}=\exc{\ball{F(u)}{a\phi(x_\lambda)}}{C}
      -a\phi(x_\lambda)   \\
      & \le& \exc{\ball{F(x_\lambda)+C}{\phi(x_\lambda)}}{C}
      -a\phi(x_\lambda)  \\
      &=&\exc{F(x_\lambda)+C}{C}+\phi(x_\lambda)-a\phi(x_\lambda) \\
      &=&\exc{F(x_\lambda)}{C}+(1-a)\phi(x_\lambda)
      =(2-a)\phi(x_\lambda).
\end{eqnarray*}
If $a>2$, the above relation already yields a contradiction. Otherwise,
by plugging $u$ in inequality $(\ref{in:EVP3})$ (remember that
$u\ne x_\lambda$), one finds
$$
  \phi(x_\lambda)<\phi(u)+(a-1)d(u,x_\lambda)\le
  \phi(x_\lambda),
$$
wherefrom contradiction arises. This fact shows that,
for the value of $\lambda$ chosen as above, it must happen that
$x_\lambda\in\Solv{F,C}$. Moreover, as a consequence of inequality
$(\ref{in:EVP2})$, by recalling formula $(\ref{eq:EVPlambda})$, one obtains
$$
   \dist{x_0}{\Solv{F,C}}\le d(x_0,x_\lambda)\le
   {\exc{F(x_0)}{C}\over a-1}.
$$
By arbitrariness of $a\in (1,\incr{F}{X})$, the last estimate completes
the proof.
\end{proof}

A generalized equation of the form $(\GE)$ can be regarded as a special
case of a more general problem, which was called set-inclusion:
given two set-valued mappings $\Psi:X\rightrightarrows Y$ and
$\Phi:X\rightrightarrows Y$ between metric spaces
$$
   \hbox{find $x\in X$ such that }  \Phi(x)\subseteq\Psi(x).
$$
A result concerning the solvability and error bounds for set-valued
inclusion problems was obtained in \cite[Theorem 3.3]{Uder17}, which
relies on the notion of set-covering mapping. Nonetheless, Theorem \ref{thm:solgloerbo}
can not be derived as a special case from the aforementioned
result. Indeed, with the identification $\Phi=F$ and $\Psi\equiv C$,
the set-covering assumption on $\Psi$ appearing in Theorem 3.3 will
be never satisfied. Moreover, in that theorem $\Phi$ is supposed to be
Lipschitz continuous on $X$ with bounded values, what is not
required in Theorem \ref{thm:solgloerbo}. In fact, even though
exploiting an analogous variational technique of proof, the approach
here proposed to address solvability and error bounds for $(\GE)$ is
based on a different behaviour of the mapping $F$: the reader should
notice that the notion of metric $C$-increase needs a partial order
structure on the range space $\Y$, whereas set-covering is a purely
metric property.

After the seminal paper \cite{Arut07}, A.V. Arutyunov and his research
group developed a theory about coincidence points for set-valued mappings,
aimed at unifying the well-known results by Banach-Caccioppoli and
Milyutin, where the property of openness plays a crucial role.
Recall that, according to \cite[First Definition]{Arut07}, a
mapping  $f:X\longrightarrow Y$ between metric spaces is said to be
open (at a linear rate) on $X$ with constant $\alpha>0$ if
$$
  f(\ball{x}{r})\supseteq\ball{f(x)}{\alpha r},\quad
  \forall r\ge 0,\ \forall x\in X.
$$
Within this theory it is possible to derive a solvability result,
with a related error bound estimate, for the following special case
of $(\GE)$:
\begin{equation}
    \hbox{find $x\in X$ such that $f(x)\in C$},
\end{equation}
where $f:X\longrightarrow\Y$ is a given single-valued mapping. More
precisely, it is possible to prove what follows: suppose that

(i) $(X,d)$ is metrically complete;

(ii) $f$ is continuous on $X$;

(iii) $f$ is open (at a linear rate) on $X$ with constant $\alpha$;

\noindent then $\Solv{f,C}\ne\varnothing$ and
$$
  \dist{x}{\Solv{f,C}}\le{\dist{f(x)}{C}\over\alpha},
  \quad\forall x\in X.
$$
Since, as mentioned above, the properties of metric $C$-increase
and openness are independent, a single-valued specialization of
Theorem \ref{thm:solgloerbo} seems not be achievable within the
theory developed after \cite{Arut07}.

In optimization contexts, where local optimality is to be investigated,
often local error bound estimates are required. Such a need motivates the
interest in a local version of the previous result. Its proof is provided
in detail inasmuch as the variational technique employed in
Theorem \ref{thm:solgloerbo} must be submitted to nontrivial
adjustments, when referred to a local setting.

\begin{theorem}[Local error bound]    \label{thm:sollocerbo}
Let  $F:X\rightrightarrows\Y$ be a set-valued mapping and let
$\bar x\in\Solv{F,C}$. Suppose that:

(i) $(X,d)$ is metrically complete;

(ii) $F$ is l.s.c. in a neighbourhood of $\bar x$ and Hausdorff
$C$-u.s.c. at $\bar x$;

(iii) $F$ is metrically $C$-increasing around $\bar x$ with
exact bound $\incr{F}{\bar x}$.

\noindent Then, for every $a\in (1,\incr{F}{\bar x})$ there exists
$\delta_a>0$ such that
$$
  \dist{x}{\Solv{F,C}}\le{\exc{F(x)}{C}\over a-1},\quad
  \forall x\in\ball{\bar x}{\delta_a}.
$$
\end{theorem}

\begin{proof}
Fixed any $a\in (1,\incr{F}{\bar x})$, by hypothesis (iii) there
exists $\delta_0>0$ such that
\begin{equation}    \label{in:metCincrFbarxdelta0}
\forall x\in\ball{\bar x}{\delta_0},\ \forall r\in (0,\delta_0),
\ \exists u\in\ball{x}{r}:\ \ball{F(u)}{ar}\subseteq\ball{F(x)+C}{r}.
\end{equation}
Consider once again the function $\phi$ defined as in $(\ref{def:excfunctFC})$.
By virtue of the lower semicontinuity of $F$ in a neighbourhood of
$\bar x$, in force of Lemma \ref{lem:s.c.Fphi}(i), there exists $\delta_1>0$
such that $\phi$ is l.s.c. at each point of $\ball{\bar x}{\delta_1}$.
Furthermore, as $F$ is also Hausdorff $C$-u.s.c. at $\bar x$, $\phi$ turns
out to be, in particular, continuous  at $\bar x$. Since $\phi(\bar x)=0$
because $\bar x\in\Solv{F,C}$, corresponding to $\delta_0$ there exists
$\delta_2>0$ such that
\begin{equation}    \label{in:phidelta02}
  \phi(x)\le {\delta_0\over 2},\quad\forall x\in\ball{\bar x}{\delta_2}.
\end{equation}
Now, let us define
$$
  \delta_*=\min\{\delta_0,\, \delta_1,\, \delta_2\}
$$
and set $\delta_a=\delta_*/4$.

To show the validity of the inequality in the thesis, fix an arbitrary
$x_0\in\ball{\bar x}{\delta_a}\backslash\Solv{F,C}$. The fact that
$x_0\not\in\Solv{F,C}$ implies $\phi(x_0)>0$.  Consider the function
$\phi:\ball{x_0}{\delta_a}\longrightarrow [0,+\infty)$. Observe that,
as it is true that $\ball{x_0}{\delta_a}\subseteq\ball{\bar x}{\delta_*/2}
\subseteq\ball{\bar x}{\delta_1}$, $\phi$ is l.s.c. on $\ball{x_0}{\delta_a}$.
This set, as a closed subset of a complete metric space, is complete.
Therefore it is possible to apply the Ekeland variational principle.
It ensures that, corresponding to $\lambda={\phi(x_0)\over a-1}$,
there exists $x_\lambda\in\ball{x_0}{\delta_a}$ such that
\begin{equation}      \label{in:EVPloc1}
    \phi(x_\lambda)\le\phi(x_0)-(a-1)d(x_\lambda,x_0),
\end{equation}
\begin{equation}     \label{in:EVPloc2}
    d(x_\lambda,x_0)\le{\phi(x_0)\over a-1}
\end{equation}
and
\begin{equation}      \label{in:EVPloc3}
   \phi(x_\lambda)<\phi(x)+(a-1)d(x,x_\lambda),\quad\forall
   x\in\ball{x_0}{\delta_a}\backslash\{x_\lambda\}.
\end{equation}
If $\phi(x_\lambda)=0$ one gets $x_\lambda\in\Solv{F,C}$ and hence,
by taking into account inequality $(\ref{in:EVPloc2})$ the thesis
is proved.
Suppose that $\phi(x_\lambda)>0$. In such an event, by inequality
$(\ref{in:EVPloc1})$ one has $\phi(x_\lambda)\le\phi(x_0)$ so,
from the fact that $x_0\in\ball{\bar x}{\delta_a}\subseteq\ball{\bar x}{\delta_2}$
one obtains
$$
  \phi(x_\lambda)\le\phi(x_0)\le{\delta_0\over 2}.
$$
On the other hand, since it holds
$$
  d(x_\lambda,\bar x)\le d(x_\lambda,x_0)+d(x_0,\bar x)
  \le{\delta_*\over 2}<\delta_0,
$$
one has $x_\lambda\in\ball{\bar x}{\delta_0}$. The last two facts make
it possible to apply inclusion $(\ref{in:metCincrFbarxdelta0})$, with
$x=x_\lambda$ and $r=\phi(x_\lambda)$, so that one gets the existence
of $u\in\ball{x_\lambda}{\phi(x_\lambda)}$ such that
$$
  \ball{F(u)}{a\phi(x_\lambda)}\subseteq\ball{F(x_\lambda)+C}
  {\phi(x_\lambda)}.
$$
Now notice that, whenever it happens that $a>2$, then by arguing in the same
way as in the proof of Theorem \ref{thm:solgloerbo}, one readily reaches a
contradiction. Thus, it is possible to assume henceforth that $a\in (1,2]$.
If it happens that $d(u,x_0)\le\delta_a$, then inequality $(\ref{in:EVPloc3})$
can be exploited. Thus, by plugging $u$ in this inequality and proceeding
exactly as done for the proof of Theorem \ref{thm:solgloerbo}, one
reaches a contradiction, thereby proving that it must be $\phi(x_\lambda)=0$.
So, in the current case, the thesis is proved.
Otherwise, if it happens that $d(u,x_0)>\delta_a=\delta_*/4$, it is
useful to observe that, on the account of inequalities $(\ref{in:EVPloc1})$
and $(\ref{in:EVPloc2})$, it holds
\begin{eqnarray*}
  d(u,x_0) &\le& d(u,x_\lambda)+d(x_\lambda,x_0)\le \phi(x_0)
  -(a-1)d(x_\lambda,x_0)+d(x_\lambda,x_0) \\
   &=& \phi(x_0)-(a-2)d(x_\lambda,x_0)\le\phi(x_0)-{a-2\over a-1}\phi(x_0)
   ={\phi(x_0)\over a-1}.
\end{eqnarray*}
Consequently, by recalling that $\bar x\in\Solv{F,C}$, one finds
$$
  \dist{x_0}{\Solv{F,C}}\le d(x_0,\bar x)\le{\delta_*\over 4}
  \le{\phi(x_0)\over a-1}.
$$
The last inequality, by arbitrariness of $x_0$, completes the proof.
\end{proof}

\vskip.5cm


\section{Applications to optimization}

\subsection{Existence of ideal efficient solutions}

Let us consider a vector optimization problem $(\VOP)$. As a
straightforward consequence of Theorem \ref{thm:solgloerbo},
the following result concerning the existence of ideal
efficient solutions to $(\VOP)$ and the distance from ideal
efficiency can be established. In its statement, by $\IESol{f}{R}$
the set of all ideal $C$-efficient solutions to $(\VOP)$
is denoted. In view of the employment of Proposition \ref{pro:addpertmetCinc},
the range space $\Y$ is supposed to be a finite-dimensional
Euclidean space, partially ordered by a closed, convex cone
$C$.

\begin{theorem}     \label{thm:vecoptexistie}
With reference to a vector optimization problem $(\VOP)$,
suppose that:

(i) $(X,d)$ is metrically complete;

(ii) $f:X\longrightarrow\R^m$ is continuous on $R$;

(iii) $-f$ is metrically $C$-increasing on $R$;

(iv) sets $R$ and $f(R)$ are closed.

\noindent Then, the set $\IESol{f}{R}$ is nonempty and closed
and it holds
$$
  \dist{x}{\IESol{f}{R}}\le{\exc{f(R)-f(x)}{C}\over \incr{-f}{R}-1},
  \quad\forall x\in R.
$$
\end{theorem}

\begin{proof}
As pointed out in Example \ref{ex:GEspecexamples}(i), $\IESol{f}{R}$
coincides with $\Solv{f(R)-f,C}$. In order to apply Theorem
\ref{thm:solgloerbo}, observe that, by virtue of hypotheses (ii) and (iv),
the set-valued mapping $F:R\rightrightarrows\R^m$, defined by
$F(x)=f(R)-f(x)$, is closed-valued and l.s.c. on $R$. This fact
enables one to invoke Proposition \ref{pro:clsolset}, which says
that $\IESol{f}{R}$ is closed. Since $-f$ is metrically $C$-increasing
on $R$ and the mapping $G:R\rightrightarrows\R^m$, given by $G\equiv
f(R)$ is evidently Lipschitz continuous with constant $\beta=0$,
Proposition \ref{pro:addpertmetCinc} ensures that $F$, as an
additive perturbation of $-f$, is still metrically $C$-increasing
on $R$, with
$$
  \incr{F}{R}\ge\incr{-f}{R}.
$$
As a closed subset of a complete metric space, $(R,d)$ is a complete
metric space.
Then, it is possible to apply Theorem \ref{thm:solgloerbo}, from which
all the remaining assertions in the thesis follow.
\end{proof}

As a comment to Theorem \ref{thm:vecoptexistie}, let us remark
that, in contrast with the most part of existence results in vector
optimization (see \cite{Jahn04,Luc89}), no kind of compactness
is assumed on $f(R)$. The existence of ideal efficient solutions
comes as a consequence of the interplay between the metric behaviour
of $f$ and the partial order $C$, and metric completeness.


\subsection{Penalization in constrained optimization}

This subsection deals with constrained (scalar) optimization
problems, whose constraint system is defined by a generalized
equation of the form $(\GE)$, namely
$$
  \min\varphi(x) \qquad\hbox{ subject to }\qquad
  F(x)\subseteq C.   \leqno  (\mathcal{P})
$$
The objective function $\varphi:X\longrightarrow\R$ is assumed
to be locally Lipschitz at $\bar x$, i.e. with the property that
there exist $\beta\ge 0$ and $\delta>0$ such that
\begin{equation}   \label{in:deflocLipvarphi}
  |\varphi(x_1)-\varphi(x_2)|\le\beta d(x_1,x_2),\quad\forall
  x\in\ball{\bar x}{\delta}.
\end{equation}
The exact local Lipschitz bound of $\varphi$ at $\bar x$ is denoted
in what follows by
$$
  \lip{\varphi}{\bar x}=\inf\{\beta\ge 0:\ \exists\delta>0\hbox{ such
  that inclusion $(\ref{in:deflocLipvarphi})$ holds}\}.
$$
Let us indicate by $\mathcal{R}$ the feasible region of $(\mathcal{P})$,
that is $\mathcal{R}=\Solv{(F,C)}$.

According to a widely-exploited scheme of analysis (see, for instance,
\cite[Ch. 6]{FacPan03}), whenever an error
bound for a given constraint system is established, a corresponding
exact penalization result is derived in a standard way. It is
worth noting that penalty method is an approach developed within
constrained optimization after \cite{Erem67,Zang67}, which consists
in reducing a given constrained optimization problem to an
unconstrained one, by replacing its objective function with a
so-called penalty functional. In the case of problem $(\mathcal{P})$,
the penalty functional is defined as
$$
   \varphi_\lambda(x)=\varphi(x)+\lambda\exc{F(x)}{C},
$$
where $\lambda$ is a positive parameter and the additional term
clearly quantifies the constraint violation.
In this context, a penalty function $\varphi_\lambda$ is said
to be exact at a local solution $\bar x\in\mathcal{R}$ to
$(\mathcal{P})$ provided that $\bar x$ is also a local solution
to the following unconstrained problem
$$
   \min_{x\in X}\varphi_\lambda(x).
   \leqno  (\mathcal{P}_\lambda)
$$
The next result provides a sufficient condition for the
exactness of a penalty function, along with a quantitative
estimate, expressed in term of problem data, for a penalty
parameter $\lambda$ to given an exact penalty function.

\begin{theorem}[Exact penalization]     \label{thm:exactpen}
Given a constrained optimization problem $(\mathcal{P})$,
suppose that:

(i) $(X,d)$ is metrically complete;

(ii) $\bar x\in\mathcal{R}$ is a local solution to $(\mathcal{P})$;

(iii) $\varphi$ is locally Lipschitz at $\bar x$;

(iv) $F$ is l.s.c. in a neighbourhood of $\bar x$ and Hausdorff
$C$-u.s.c. at $\bar x$;

(v) $F$ is metrically $C$-increasing around $\bar x$.

\noindent Then, for every $\lambda>{\lip{\varphi}{\bar x}\over\incr{F}{\bar x}-1}$,
the function $\varphi_\lambda$ is exact at $\bar x$.
\end{theorem}

\begin{proof}
Fix $\lambda>{\lip{\varphi}{\bar x}\over\incr{F}{\bar x}-1}$. Then, it is
possible to pick constants $a$, $\tilde a$ and $\beta$, with
$$
  \beta>\lip{\varphi}{\bar x} \quad\hbox{ and }\quad
  1<\tilde a <a <\incr{F}{\bar x},
$$
in such a way that
\begin{equation}    \label{in:lambdatildeaalip}
   \lambda>{\beta\over\tilde a-1}>{\beta\over a-1}>
   {\lip{\varphi}{\bar x}\over\incr{F}{\bar x}-1}.
\end{equation}
As a consequence of hypothesis (ii), there exists $\delta_0$
such that
\begin{equation}     \label{in:locoptimbarx}
   \varphi(x)\ge\varphi(\bar x),\quad\forall x\in
   \ball{\bar x}{\delta_0}\cap\mathcal{R}.
\end{equation}
By virtue of hypothesis (iii), there exists $\delta_\beta>0$
such that
\begin{equation}     \label{in:loclipvarphibarx}
  |\varphi(x_1)-\varphi(x_2)|\le\beta d(x_1,x_2),\quad\forall
  x\in\ball{\bar x}{\delta_\beta}.
\end{equation}
Since, under the hypotheses (i), (iv) and (iv) it is possible to apply
Theorem \ref{thm:sollocerbo}, there exists $\delta_a>0$ such that
\begin{equation}    \label{in:locerboadeltaa}
    \dist{x}{\mathcal{R}}\le{\exc{F(x)}{C}\over a-1},\quad
    \forall x\in\ball{\bar x}{\delta_a}.
\end{equation}
Now, define
\begin{equation}    \label{eq:defr4delta}
   r_*={1\over 4}\min\{\delta_0,\, \delta_\beta,\, \delta_a\}.
\end{equation}
According to the definition of $\varphi_\lambda$, clearly it is
$$
  \varphi_\lambda(x)=\varphi(x)\ge\varphi(\bar x)=\varphi_\lambda(\bar x),
  \quad\forall x\in\ball{\bar x}{r_*}\cap\mathcal{R}.
$$
It remains to prove that $\varphi_\lambda$ fulfils a similar
inequality also for every $x\in\ball{\bar x}{r_*}\backslash
\mathcal{R}$. To this aim, fix an arbitrary $x\in
\ball{\bar x}{r_*}\backslash\mathcal{R}$. Since $r_*<\delta_a$, inequality
$(\ref{in:locerboadeltaa})$ is valid at each point of $\ball{\bar x}{r_*}$
and hence
$$
  \dist{x}{\mathcal{R}}<{\exc{F(x)}{C}\over \tilde a-1}.
$$
Remember that, as $x\not\in\mathcal{R}$, it must be
$\exc{F(x)}{C}>0$). Moreover, under the aforementioned hypothesis
of lower semicontinuity on $F$, in the light of Proposition
\ref{pro:clsolset} $\mathcal{R}$ turns out to be locally closed around $\bar x$
and hence one has also $\dist{x}{\mathcal{R}}>0$.
These facts imply the existence of $z\in\mathcal{R}$
such that
\begin{equation}     \label{in:distxz2}
  d(x,z)<2\dist{x}{\mathcal{R}}
\end{equation}
and
\begin{equation}     \label{in:exctildea}
  d(x,z)<{\exc{F(x)}{C}\over \tilde a-1}.
\end{equation}
Notice that, since it is $\bar x\in\mathcal{R}$, then by inequality
$(\ref{in:distxz2})$ one gets
$$
d(x,z)<2d(x,\bar x)\le 2r_*.
$$
Owing to the definition of $r_*$ made in $(\ref{eq:defr4delta})$, it follows
$$
  d(z,\bar x)\le d(z,x)+d(x,\bar x)\le 3r_*<\min
  \{\delta_0,\, \delta_\beta,\, \delta_a\}.
$$
As a consequence, one obtains
$$
  z\in\mathcal{R}\cap\ball{\bar x}{\delta_0}\cap\ball{\bar x}{\delta_\beta}.
$$
This fact enables one to apply inequalities $(\ref{in:locoptimbarx})$ and
$(\ref{in:loclipvarphibarx})$. Thus, by recalling inequality $(\ref{in:exctildea})$,
one finds
$$
  \varphi(x)\ge\varphi(z)-\beta d(x,z)\ge\varphi(\bar x)-\beta
  {\exc{F(x)}{C}\over \tilde a-1},
$$
wherefrom, in the light of inequality $(\ref{in:lambdatildeaalip})$, one
deduces
$$
   \varphi_\lambda(x)\ge\varphi(\bar x)=\varphi_\lambda(\bar x).
$$
The last inequality completes the proof.
\end{proof}

Theorem \ref{thm:exactpen} can be used as a starting point for deriving
necessary optimality conditions, when problem $(\mathcal{P})$ is considered
in more structured settings. For instance, if $F:\X\rightrightarrows\Y$ is
a set-valued mapping between normed spaces with the property
$$
  F(tx_1+(1-t)x_2)\subseteq tF(x_1)+(1-t)F(x_2),\quad\forall
  x_1,\, x_2\in\X,\ \forall t\in [0,1]
$$
(this happens, for example, whenever $F$ is defined by a set of linear operators),
then the function $x\mapsto\exc{F(x)}{C}$ turns out to be convex.
Thus, under the assumptions of Theorem \ref{thm:exactpen}, a necessary
optimality condition can be readily expressed in terms of the Clarke subdifferential
(see \cite[Ch. 10]{Clar13}).
In other cases, to derive a verifiable necessary optimality condition,
it may be useful to handle the term $x\mapsto\exc{F(x)}{C}$ by other kinds of
subdifferential (sometimes more elaborated), which are available in nonsmooth analysis
(see \cite{BorZhu05,Mord06,Shir07,Peno13}). It is worth mentioning that
some of them may perform better than the Clarke subdifferential, even in the
Lipschitz case, inasmuch as they lead to smaller constructions in
the dual space.

Results of exact penalization are often complemented with conditions upon which
from the global strict optimality of an element for a penalized problem
one gets its (global) optimality for the
original problem. A similar result can be achieved also
in the context of problem $(\mathcal{P})$, by employing this time
the global error bound estimate for generalized equations $(\GE)$.

\begin{proposition}
With reference to a problem $(\mathcal{P})$, suppose that:

(i) $(X,d)$ is metrically complete;

(ii) $(\mathcal{P})$ admits global solutions;

(iii) $\varphi$ is Lipschitz continuous on $X$;

(iv) $F$ is l.s.c. on $X$;

(v) $F$ is metrically $C$-increasing on $X$.

\noindent If $\hat x\in X$  is a strict global solution to problem
$(\mathcal{P}_{\lambda_\epsilon})$, with
$$
  \lambda_\epsilon={(1+\epsilon)\beta\over\incr{F}{X}-1},
$$
for some $\epsilon>0$, then $\hat x$ is a global solution also
to $(\mathcal{P})$.
\end{proposition}

\begin{proof}
Once proved that $\hat x\in\mathcal{R}$ the thesis follows at once, so
assume, ab absurdo, that $\hat x\not\in\mathcal{R}$. Under the above hypotheses
it is possible to apply Theorem \ref{thm:solgloerbo} and, in
particular, the error bound estimate in its thesis. This amounts to
say that, corresponding to $\epsilon>0$, there exists $x_\epsilon
\in\mathcal{R}$ such that
\begin{equation}     \label{in:epsilongloerbohatx}
    d(x_\epsilon,\hat x)\le {(1+\epsilon)\exc{F(\hat x)}{C}\over
    \incr{F}{X}-1}.
\end{equation}
According to hypothesis (ii), let $\bar x\in\mathcal{R}$ be a global
solution to $(\mathcal{P})$. From the global optimality of $\hat x$
for $(\mathcal{P}_{\lambda_\epsilon})$, inequality $(\ref{in:epsilongloerbohatx})$
and hypthesis (iii), it follows
\begin{eqnarray*}
   \varphi(\bar x) &=&\varphi_{\lambda_\epsilon}(\bar x)\ge
   \varphi_{\lambda_\epsilon}(\hat x)=\varphi(\hat x)+\lambda_\epsilon
   \exc{F(\hat x)}{C}  \\
   &\ge& \varphi(x_\epsilon)-\beta d(\hat x,x_\epsilon)+
   \lambda_\epsilon\exc{F(\hat x)}{C}\ge\varphi(x_\epsilon) \\
   &\ge& \varphi(\bar x).
\end{eqnarray*}
The above inequalities imply that $\varphi_{\lambda_\epsilon}(\hat x)=
\varphi_{\lambda_\epsilon}(x_\epsilon)$. Since $\hat x$ has been supposed
to be a strict global solution to $(\mathcal{P}_{\lambda_\epsilon})$, the
last equality forces $\hat x=x_\epsilon$, but this contradicts the fact
that $\hat x\not\in\mathcal{R}$. The proof is complete.
\end{proof}

\vskip1cm

\noindent {\bf Acknowledgements} The author thanks the anonymous referees
for valuable comments.



\end{document}